\documentclass[12pt]{article}

\usepackage[paper=letterpaper,margin=1in,twoside=false,includehead,headheight=18pt]{geometry}

\usepackage[utf8]{inputenc}

\usepackage{amsmath,amssymb,amsthm} 
\usepackage{enumerate,paralist}
\usepackage{graphicx}
\usepackage{mathtools}

\usepackage[draft]{fixme}

\newtheorem{theorem}{Theorem}[section]
\newtheorem{lemma}[theorem]{Lemma} 
 
\newtheorem{corollary}[theorem]{Corollary}
\newtheorem{porism}[theorem]{Porism}
\newtheorem*{main}{Main Theorem}
\newtheorem*{maincor}{Main Corollary}

{\theoremstyle{definition} 

\newtheorem*{definition}{Definition} 
\newtheorem{conjecture}{Conjecture} 
 
 }

\newtheoremstyle{named}%
  {}{}						
  {\upshape}				
  {0pt}{\bfseries}			
  {.}						
  {.5em}					
  {\thmname{#1}\thmnote{ #3}}  

\theoremstyle{named}

\newcommand{\wo}{\setminus}

\newcommand{\card}{\#}
\newcommand{\set}[1]{\left\{{#1}\right\}} 
\newcommand{\setof}[2]{\left\{{#1}\,:\,{#2}\right\}}
\newcommand{\of}{\subseteq}

\newcommand{\intersection}{\bigcap}
\newcommand{\union}{\bigcup}

\newcommand{\0}{\emptyset}

\newcommand{\N}{\mathbb{N}}

\renewcommand{\complement}{\overline}
\renewcommand{\d}{\delta}
\newcommand{\D}{\Delta}

\newcommand{\ind}{i}
\newcommand{\Ind}{{\mathcal I}}

\renewcommand{\k}{k}
\newcommand{\K}{{\mathcal K}}

\newcommand{\w}{w}

\newcommand{\Gp}[1]{G_{#1}}

\newcommand{\fl}{\phi}
\DeclarePairedDelimiter{\abs}{\lvert}{\rvert}

\begin{document}

\title{The maximum number of complete subgraphs in a graph with given maximum degree}
\date{\today}
\author{Jonathan Cutler\\
\small{Montclair State University}\\
\small{\texttt{jonathan.cutler@montclair.edu}}
 \and 
A.J.~Radcliffe\\
\small{University of Nebraska-Lincoln}\\
\small{\texttt{aradcliffe1@math.unl.edu}}}
\maketitle
\begin{abstract}
	Extremal problems involving the enumeration of graph substructures have a long history in graph theory.  For example, the number of independent sets in a $d$-regular graph on $n$ vertices is at most $(2^{d+1}-1)^{n/2d}$ by the Kahn-Zhao theorem \cite{K01,Z}.  Relaxing the regularity constraint to a minimum degree condition, Galvin \cite{G} conjectured that, for $n\geq 2d$, the number of independent sets in a graph with $\delta(G)\geq d$ is at most that in $K_{d,n-d}$.  
	
	In this paper, we give a lower bound on the number of independent sets in a $d$-regular graph mirroring the upper bound in the Kahn-Zhao theorem.  The main result of this paper is a proof of a strengthened form of Galvin's conjecture, covering the case $n\leq 2d$ as well.  We find it convenient to address this problem from the perspective of $\complement{G}$.  In other words, we give an upper bound on the number of complete subgraphs of a graph $G$ on $n$ vertices with $\Delta(G)\leq r$, valid for all values of $n$ and $r$.
\end{abstract}

\maketitle

\section{Introduction} 
\label{sec:introduction}

There has been quite a bit of recent interest in a range of extremal problems involving counting the number of a given type of substructure in a graph.  For instance, the number of independent sets or the number of the complete subgraphs\footnote{We use the term \emph{clique} to refer to a complete subgraph, not necessarily a maximal complete subgraph.} of a graph.  We let $\Ind(G)$ be the set of independent sets in the graph $G$ and $\K(G)$ be the set of cliques in $G$.  We write $\ind(G)$ and $\k(G)$ for $\abs{\Ind(G)}$ and $\abs{\K(G)}$, respectively.

A classic example of type of result we consider is the Kahn-Zhao theorem, proved first for bipartite graphs by Kahn \cite{K01} and later extended to all graphs by Zhao \cite{Z}.

\begin{theorem}[Kahn-Zhao]\label{thm:KZ}
	If $G$ is a $d$-regular graph with $n$ vertices then
	\[
		\ind(G)^{\frac1n} \le \ind(K_{d,d})^{\frac1{2d}}=(2^{d+1}-1)^{\frac{1}{2d}}.
	\]
\end{theorem}

The Kahn-Zhao theorem is tight when $n$ is a multiple of $2d$ with $\frac{n}{2d}K_{d,d}$, i.e., $\frac{n}{2d}$ copies of $K_{d,d}$, achieving equality in the bound.  Little is known about extremal examples when $2d$ does not divide $n$.  One result of this paper is a corresponding theorem proving 
\[
	\ind(G)^{1/n}\geq \ind(K_{d+1})^{1/(d+1)}
\]
for $G$ a $d$-regular graph on $n$ vertices.

Extremal enumeration problems for complete subgraphs run in parallel to those for independent sets.  If $G$ is a graph with $N$ independent sets, then $\complement{G}$ is a graph with $N$ cliques.  Any degree condition on $G$ translates into a corresponding degree condition on $\complement{G}$.  For example, the Kahn-Zhao theorem can be rephrased as follows: if $G$ is an $r$-regular graph on $n$ vertices, then 
\[
	\k(G)^{\frac{1}n}\leq k(2K_{n-1-r})^{\frac{1}{2(n-1-r)}}.
\]
Note, however, that in the intuitively natural regime where $r$ is fixed and $n$ is large, we do not expect this bound to be tight.

Although regularity is a very natural condition to impose, a range of other conditions have been studied.  For instance, it is a consequence of the Kruskal-Katona theorem \cite{K63,K68} that among all graphs of given average degree, the lex graph\footnote{The \emph{lex graph with $n$ vertices and $m$ edges}, denoted $L(n,m)$, has vertex set $[n]=\set{1,2,\ldots,n}$ and edge set an initial segment of size $m$ in $\binom{[n]}{2}$ according to the lexicographic order.} has the largest number of independent sets, indeed the largest number of independent sets of any fixed size.  For a derivation, see, e.g., \cite{CR11}.  Another example is the oft-rediscovered result originally due to Zykov \cite{Z49} (see also \cite{E62, S71, H76, R}) which bounds the number of cliques in graphs with bounded clique number, $\omega(G)$.

\begin{theorem}[Zykov]
	If $G$ is a graph with $n$ vertices and $\omega(G)\leq \omega$, then 
	\[
		\k(G)\leq \k(T_{n,\omega}),
	\]
	where $T_{n,\omega}$ is the Tur\'an graph with $\omega$ parts.
\end{theorem}

Equivalently, this gives a bound on $i(G)$ for graphs with bounded independence number; the extremal graph is a union of disjoint complete graphs of almost equal sizes.

The main problem we discuss in this paper is that of computing, given $n$ and $r$,
\[
	\max\setof{\k(G)}{\text{$G$ is a graph on $n$ vertices with $\D(G)\le r$}}.
\]
This problem is equivalent to that of determining
\[
	\max\setof{\ind(G)}{\text{$G$ is a graph on $n$ vertices with $\d(G)\ge d$}},
\]
where $d=n-1-r$.  Galvin \cite{G} made the following conjecture.

\begin{conjecture}[Galvin]\label{conj:Galvin}
	If $G$ is a graph on $n$ vertices with minimum degree at least $d$, where $n\geq 2d$, then $i(G)\leq i(K_{d,n-d})$.
\end{conjecture}

Galvin proved in \cite{G} that the conjecture holds for fixed $d$ and $n$ sufficiently large, and also for all $n$ with $d=1$.  Engbers and Galvin \cite{EG} proved the conjecture when $d=2$ or $3$.  Alexander and Mink together with the first author \cite{ACM} proved the conjecture for bipartite graphs.

The main result of this paper is to prove a strengthened version of Conjecture~\ref{conj:Galvin}.  It is convenient for us to phrase our theorem in the language of cliques.  In the process of taking complements, the minimum degree condition is replaced by a maximum degree condition.  From this perspective, we are able to remove the condition from the conjecture relating $n$ and $d$.  We prove the following theorem, and hence the subsequent corollary.

\begin{main}
	For all $n,r\in \N$, write $n=a(r+1)+b$ with $0\le b\le r$. If $G$ is a graph on $n$ vertices with $\Delta(G)\leq r$, then 
	\[
		\k(G) \leq \k(aK_{r+1} \cup K_b),
	\]
	with equality if and only if $G=aK_{r+1}\cup K_b$, or $r=2$ and $G=(a-1)K_3\cup C_4$ or $(a-1)K_3\cup C_5$.
\end{main}

\begin{maincor}
	For all $n,d\in \N$, write $n=a(n-d)+b$ with $0\le b< n-d$. If $G$ is a graph on $n$ vertices with $\delta(G)\geq d$, then 
	\[
		\ind(G) \leq \ind\left(\,\complement{aK_{n-d} \cup K_b}\,\right)=a(2^{n-d}-1)+2^b.
	\]
\end{maincor}

In the low density regime, i.e., for $n\geq 2d$, the extremal graph is indeed $K_{d,n-d}$.  For the high density regime, where $a\geq 2$, the extremal graphs become more and more regular and we believe that this bound is the best known even in the regular case.

In Section~\ref{sec:signposts} we prove a weak form of the Main Theorem, valid only for $n$ a multiple of $r+1$, proved already by Engbers and Galvin \cite{EG}.  There is a striking similarity of proof technique between this result and the lower bound mentioned above on $i(G)^{1/n}$ for $G$ a regular graph, so we include both in the same section.  The argument for the weak bound essentially forms the kernel of the proof of the Main Theorem.  We outline that proof in Section~\ref{sec:outline}.  In Sections~\ref{sec:FL}, \ref{sec:strong}, and \ref{sec:discharging}, we introduce the main tools of the proof.  Finally, in Section~\ref{sec:proof}, we prove the Main Theorem.


\section{Weak bounds} 
\label{sec:signposts}

In the section we prove two simple results that illustrate some of the methods we will use later in the paper. The first is a Kahn-Zhao type result (though substantially easier to prove) concerning the minimum number of independent sets in a $d$-regular graph.  We let $\Ind_t(G)$ be the set of independent sets of size $t$ in $G$ and $\ind_t(G)=\abs{\Ind_t(G)}$.

\begin{theorem}\label{thm:ind_signpost}
	If $G$ is $d$-regular on $n=a(d+1)$ vertices, then
	\[
		\ind(G) \ge \ind(aK_{d+1}) = (d+2)^a.
	\]
	Indeed, for all $0\leq t\leq n$ we have $\ind_t(G)\ge \ind_t(aK_{d+1})=(d+1)^t\binom{a}t$.
\end{theorem}

The essential part of the proof of this theorem is contained in the following lemma.

\begin{lemma}\label{lem:ind_upwards}
	If $G$ is $d$-regular on $n=a(d+1)$ vertices then for all $1\leq t\leq n$ and $I\in \Ind_{t-1}(G)$ we have
	\[
		\card\setof{J\in\Ind_t(G)}{J\supseteq I} \ge (a-t+1)(d+1).
	\]
\end{lemma}

\begin{proof}
	The number on the left is exactly the number of common non-neighbors of $I$ (other than the elements of $I$). This is exactly
	\[
		\abs[\Big]{V(G) - \union_{x\in I} N[x]} \ge n - (t-1) (d+1) = (a-t+1)(d+1), 
	\]
	where $N[x]$ is the closed neighborhood of $x$.
\end{proof}

\begin{proof}[Proof of Theorem~\ref{thm:ind_signpost}]
	We'll prove the stronger statement by induction. Certainly, since $e(G) = e\left(aK_{d+1}\right)$ the statement is true for $t=2$.  (The statement is trivial for $t=0,1$.) Suppose now that $t>2$. By double-counting, we have
	\begin{align*}
		\ind_t(G) &= \frac1t \sum_{I\in \Ind_{t-1}(G)} \card\setof{J\in\Ind_t(G)}{J\supseteq I} \\
			&\ge \frac1t  (a-t+1)(d+1)\, i_{t-1}(G)\\
			&\ge \frac{(a-t+1)(d+1)}t\, i_{t-1}(aK_{d+1})  \\
			&= \ind_t(aK_{d+1}).\qedhere
	\end{align*}
\end{proof}

\begin{corollary}\label{cor:minind}
	If $G$ is $d$-regular on $n$-vertices then 
	\[
		\ind(G)^{\frac{1}{n}} \ge \ind(K_{d+1})^{\frac1{d+1}} = (d+2)^{\frac1{d+1}}
	\]
\end{corollary}

\begin{proof}
\[
	\ind(G)^{d+1} = \ind((d+1)G) \ge \ind(nK_{d+1}) = \ind(K_{d+1})^n.\qedhere
\]
\end{proof}

Noting that our proof of Lemma~\ref{lem:ind_upwards} only required an upper bound on the degrees of vertices in $G$, we also get the following result.

\begin{porism}
	If $G$ is a graph on $n$ vertices with $\Delta(G)\le d$ then
	\[
		\ind(G)^{\frac1n} \ge \ind(K_{d+1})^{\frac{1}{d+1}}.
	\]
\end{porism}

Our second ``signpost'' result is a best possible bound on $\k(G)$ for graphs with $\Delta(G)\leq r$, valid only when $r+1$ divides $n(G)$.  This result appears in a paper of Engbers and Galvin \cite{EG}, but it's important for the development of the rest of the paper that we include the proof.  We let $\K_t(G)$ be the set of cliques of $G$ of size $t$ and set $\k_t(G)=\abs{\K_t(G)}$.

\begin{theorem}[Engbers, Galvin]\label{thm:k_signpost}
	If $G$ is a graph on $n=a(r+1)$ vertices and $\Delta(G)\leq r$, then
		\[
			\k(G) \le \k\bigl(a K_{r+1}\bigr) = 1+a \bigl(2^{r+1}-1\bigr).
		\]
		Indeed, for $0\leq t\leq n$, we have $\k_t(G)\leq \k_t(aK_{r+1})$.
\end{theorem}

\begin{proof}
	As in the proof of Theorem~\ref{thm:ind_signpost}, we start by proving that if $C\in \K_{t-1}(G)$, 
	\[
		\card\setof{D\in \K_t(G)}{D\supseteq C}\leq r-t+2.
	\]
	This is immediate since each vertex in $C$ has at most $r-t+2$ neighbors outside $C$.  Thus, by induction on $t$ (starting at $t=1$),
	\begin{align*}
		\k_t(G)&\leq \frac{r-t+2}{t}\,\k_{t-1}(G)\\
		&\leq \frac{r-t+2}{t}\,\k_{t-1}(aK_{r+1})\\
		&= \frac{r-t+2}{t}\,a\binom{r+1}{t-1}\\
		&= a\binom{r+1}t=\k_t(aK_{r+1}).\qedhere
	\end{align*} 
\end{proof}

In the remaining sections of the paper, we prove a best possible bound for all $n$ on the number of complete subgraphs of a graph with $\Delta(G)\leq r$.


\section{Outline of the proof} 
\label{sec:outline}

Our approach to the proof of the Main Theorem is as follows. We consider a graph $G$ on $n$ vertices with $\D(G)\le r$.  If $K_{r+1}\of G$ then we are done by induction.  Motivated by the proof of Theorem~\ref{thm:k_signpost}, we will assign weights to the complete subgraphs of $G$: if $C\in \K(G)$ then we set
\[
	\w(C) = \abs[\Big]{\intersection_{x\in C} N(x)},
\]
the number of common neighbors of all the vertices in $C$.  (In particular of course no element of $C$ is counted since it is not adjacent to itself.)  Equivalently, $\w(C)$ is the number of cliques of $G$ of size $|C|+1$ containing $C$.  By the same double-counting argument as in Theorem~\ref{thm:ind_signpost} we have that
\begin{equation}\label{eq:kbound}
	\k_{t}(G) = \frac1{t} \sum_{C\in \K_{t-1}(G)} \w(C).
\end{equation}
Thus if the average weight of $(t-1)$-cliques is small then there will not be many $t$-cliques in total. The bound on $\D(G)$ shows that $\w(C) \le r+1-\abs{C}$ for all $C$. If a clique $C$ satisfies this bound with equality then we call it \emph{tight}. The core of our proof is to focus on the tight cliques.  Suppose then that $G$ is a graph with $\Delta(G)\leq r$ and $n=a(r+1)+b$ vertices.  In crude outline our proof has four cases:
\begin{enumerate}[I)]
    \item $G$ contains a $K_{r+1}$, in which case we are done by induction.
    \item $G$ has no tight cliques. In this case, by (\ref{eq:kbound}), we observe that $G$ satisfies the \emph{strong inequalities}: for all $t\geq 3$ we have
    \[
        \k_{t}(G) \le \frac {r-t+1}{t} \;\k_{t-1}(G).
    \] 
	We will show that in all nontrivial cases, the strong inequalities imply that $\k(G)<\k(aK_{r+1}\cup K_b)$.
    \item $G$ has some tight clique for which a certain parameter, that we call \emph{fixed loss}, is small.  In this case, we modify the graph $G$ to obtain a graph $G'$ with $\k(G')>\k(G)$.
    \item $G$ has tight cliques each having large fixed loss.  We prove that $G$ satisfies the strong inequalities despite having tight cliques.\label{alligator}
\end{enumerate}

In the next section, we introduce and discuss fixed loss.  In subsequent sections, we consider the strong inequalities and use a discharging technique to deal with case \ref{alligator}.


\section{Fixed loss} 
\label{sec:FL}

The modification we hope to do to a graph containing a tight clique is relatively simple. It is described in the following definition.

\begin{definition}
	Suppose that $G$ is a graph with $\D(G)\le r$ and $T\of V(G)$ is a tight clique of size $t$. We let $S_T = \intersection_{x\in T} N(x)$ and define a new graph by converting $T\cup S_T$ into a clique (of size $r+1$) and deleting all the edges $[S_T,V(G)\wo(T\cup S)]$. In other words we define
	\[
		\Gp{T} = G + \binom{S_T}{2} - [S_T,V(G)\wo(T\cup S_T)],
	\]
where $\binom{S_T}2$ is the set of all pairs in $S_T$ and, for sets $U$ and $V$, $[U,V]=\setof{uv}{u\in U, v\in V}$.
\end{definition}

\begin{lemma}\label{lem:dR}
	With the notation of the previous definition, for all $x\in S$ we have
	\[
		\abs{N_G(x)\wo (T\cup S)} \le d_{R_T}(x)
	\]
\end{lemma}

\begin{proof}
	Immediate.
\end{proof}

If $G$, $T$ and $S$ are as in the definition, then there are no edges in $G$ between $T$ and $V(G)\wo (T\cup S)$. Thus $\Gp{T}$ contains a copy of $K_{r+1}$ on $T\cup S$. By Lemma~\ref{lem:dR}, we have $\Delta(\Gp{T}) \le r$. Vertices in $T\cup S$ now have degree exactly $r$ and no other vertex has had its degree increased. We will give a bound on $\k(\Gp{T})$ that is described in terms of the edges inside $S$ that we are ``filling in''.

\begin{definition}
	Let $G$ be a graph with $\D(G)\le r$ and $T\of V(G)$ a tight clique in $G$. Set $S = S_T$. Then we define
	\[
		R_T = \complement{G[S]},
	\]
i.e., the graph on $S$ whose edges are those not in $G$.  If $R$ is any graph and $I\of V(R)$ we define
	\[
		\d_I = \min\setof{d_R(x)}{x\in I}.
	\]
	We define the \emph{fixed loss of a graph $R$} to be
	\[
		\fl(R) = \sum_{\substack{I\in \Ind(R)\\I\not=\0}} (2^{\d_I} - 1).
	\]
\end{definition}

The following lemma gives a lower bound on $\k(G_T)$ in terms of $\fl(R)$.

\begin{lemma}\label{lem:kGS}
	If $G$ is a graph with $\D(G)\le r$ and $T\of V(G)$ is a tight $t$-clique then
	\[
		\k(G_T) \ge \k(G) + 2^{r+1} - 2^t\,\ind(R_T) - \fl(R_T).
	\]
\end{lemma}

\begin{proof}
	For convenience we will set $S = S_T$ and $V'=V(G)\wo (T\cup S)$. We also abbreviate $R_T$ to $R$. We will show that 
	\begin{align*}
		\abs{\K(G)\wo \K(G_T)} &\le \fl(R) \text{ and} \\
		\abs{\K(G_T)\wo \K(G)} &= 2^{r+1} - 2^t\,(\ind(R)-s-1),
	\end{align*}
	from which the result follows immediately. Consider first a clique $C\in \K(G)\wo \K(G_T)$. It must meet both $S$ and $V'$. We will count such cliques according to their intersection with $S$, so set $I=C\cap S$. This intersection must be an independent set in $R$ (since edges of $R$ are missing in $G$) and (at a bare minimum) the elements of $C\cap V'$ must be common neighbors of all the elements of $I$. Since, by Lemma~\ref{lem:dR}, each $x\in I$ has at most $d_{R}(x)$ neighbors in $V'$, we can fix some $x_0\in I$ with $d_{R}(x_0)=\d_I$ and we see that each such $C$ is associated with a unique non-empty subset of $N_G(x_0)\cap V'$. Thus there are at most $2^{\d_I}-1$ such $C$, and at most $\fl(R)$ cliques in $\K(G)\wo \K(G_T)$ in total.
	
	Turning now to $\K(G_T)\wo \K(G)$ we see that if $C\in \K(G_T)\wo \K(G)$ then we must have $C\of T\cup S$ with $C\cap S\neq \0$. All such subsets are cliques of $G_T$. The ones that are cliques of $G$ are those not missing an edge in $S$, i.e., those that meet $S$ in an independent set of $R$ of size at least two. 
\end{proof}

\begin{corollary}\label{cor:fli}
	With the setup of Lemma~\ref{lem:kGS} and setting $s=\abs{S}$, if 
	\[
		2^t > \frac{\fl(R)}{2^s-\ind(R)+s+1}
	\]
	then $\k(G_T)>\k(G)$.
\end{corollary}

We start our investigation of the graph parameter $\fl$ by proving some simple (if somewhat surprising) extremal results.

\begin{theorem}\label{thm:maxFL}
	If $R$ is a graph on $s$ vertices then
	\[
		\fl(R) \le \fl(K_s) = s(2^{s-1}-1).
	\]
\end{theorem}

\begin{proof}
	We will in fact prove something stronger, that 
	\[
		\sum_{\substack{I\in \Ind(R)\\I\not=\0}} \abs{I}\bigl(2^{\d_I}-1\bigr) \le s(2^{s-1}-1).
	\]
	We calculate as follows.
	\begin{align*}
		\sum_{\substack{I\in \Ind(R)\\I\not=\0}} \abs{I}\bigl(2^{\d_I}-1\bigr) 
			&= \sum_{\substack{I\in \Ind(R)\\I\not=\0}} \sum_{x\in I} \bigl(2^{\d_I}-1\bigr) \\
			&\le \sum_{\substack{I\in \Ind(R)\\I\not=\0}} \sum_{x\in I} \bigl(2^{d(x)}-1\bigr) \\
			&= \sum_{x\in V(R)} \sum_{\substack{I \in \Ind(R)\\ x\in I}} \bigl(2^{d(x)}-1\bigr) \\
			&\le \sum_{x\in V(R)} 2^{s-d(x)-1} \bigl(2^{d(x)}-1\bigr) \\
			&=\sum_{x\in V(R)} (2^{s-1}-2^{s-d(x)-1})\\
			&\le s(2^{s-1}-1).
	\end{align*}
	In the antepenultimate step we used the fact that if $I$ is independent and contains $x$ then certainly $I\wo \set{x} \of V(R)\wo N[x]$. 
\end{proof}

We will also require a slightly more technical result bounding $\fl(R)$ in terms of both $s$ and the number of vertices of $R$ of degree one.  Before we do this, we need to start with a simple lemma showing that we can assume that $R$ contains no $K_2$ components.

\begin{lemma}\label{lem:k2}
	Suppose that $G$ is a graph with $\Delta(G)\leq r$ and $T$ is a tight clique in $G$ with $\abs{T}\geq 2$.  If $R=R_T$ contains a $K_2$ component on vertices $u$ and $v$, then the graph $G'$ obtained from $G$ by adding the edge $uv$ and deleting any edges in $[\set{u,v},V(G)\wo (T\cup S_T)]$ has $\Delta(G')\leq r$ and $\k(G')>\k(G)$.
\end{lemma}

\begin{proof}
	Most cliques are the same in $G$ and $G'$.  In $G'$ we no longer have the $K_2$s corresponding to edges in $[\set{u,v},V(G)\wo (T\cup S_T)]$; since each of $u$ and $v$ is incident to at most one such edge (by Lemma~\ref{lem:dR}), we have lost at most two cliques.  On the other hand, we have gained the edge $uv$ and $t\geq 2$ triangles of the form $\set{x,u,v}$ with $x\in T$.
\end{proof}

\begin{theorem}\label{thm:flell}
	Let $R$ be a graph on $s$ vertices having $\ell$ vertices of degree one and containing neither a $K_1$ nor a $K_2$ component.  Then
	\[
		\fl(R)\leq 2^{s}+(s-\ell-2)2^{s-\ell-1}.
	\]
\end{theorem}

\begin{proof}
	Let $L$ be the set of vertices of degree one.  We split up the sum computing $\fl(R)$ into two parts, the contributions of independent sets containing an element of $L$ and the rest.  To this end, let
	\begin{align*}
		\fl'(R)&=\sum_{\substack{I\in \Ind(R)\\I\cap L\neq \0}} 2^{\d_I}-1=\card\setof{I\in \Ind(R)}{I\cap L\neq \0},\quad\text{and}\\
		\fl''(R)&=\sum_{\substack{\0\neq I\in \Ind(R)\\I\cap L=\0}} 2^{\d_I}-1.
	\end{align*}
	To bound the first term, we observe
	\[
		\fl'(R)=\card\setof{I\in \Ind(R)}{I\cap L\neq \0}\leq (2^{\ell}-1)2^{s-\ell-1}.
	\]
	This follows from the fact that no vertex of $L$ is adjacent to any other and therefore, given any nonempty subset $L'$ of $L$, at least one vertex of $R\wo L$ is excluded from $I$.  So there are at most $2^{s-\ell-1}$ independent sets contributing to $\fl'(R)$ of the form $L'\cup J$ where $L\cap J=\0$.  On the other hand, writing $d_{\ell}(v)$ for $\abs{N(v)\cap L}$,
	\begin{align*}
		\fl''(R)&=\sum_{\substack{\0\neq I\in \Ind(R)\\I\cap L=\0}} 2^{\d_I}-1\\
		&\leq \sum_{\substack{\0\neq I\in \Ind(R)\\I\cap L=\0}} \abs{I}(2^{\d_I}-1)\\
		&= \sum_{v\in V(R)}\; \sum_{\substack{v\in I\in \Ind(R)\\I\cap L=\0}} 2^{\d_I}-1\\
		&\leq \sum_{v\in V(R)}\; \sum_{\substack{v\in I\in \Ind(R)\\I\cap L=\0}} 2^{d(v)}-1\\
		&\leq \sum_{v\in V(R)} 2^{s-\ell-d(v)+d_{\ell}(v)-1}(2^{d(v)}-1)\\
		&= \sum_{v\in V(R)} 2^{s-\ell+d_{\ell}(v)-1}-2^{s-\ell-d(v)+d_{\ell}(v)-1}\\
		&= 2^{s-\ell-1}\left(\sum_{v\in V(R)} 2^{d_{\ell}(v)}-\sum_{v\in V(R)}2^{d_{\ell}(v)-d(v)}\right)\\
		&\leq 2^{s-\ell-1}(2^{\ell}+s-\ell-1).
	\end{align*}
The fifth step above follows as in the proof of Theorem~\ref{thm:maxFL} and the final step uses the convexity of $2^x$ on the first term and ignores the second.

Combining these bounds, we have
\[
	\fl(R)=\fl'(R)+\fl''(R)\leq 2^s+(s-\ell-2)2^{s-\ell-1}.\qedhere
\]
\end{proof}


\section{The strong inequalities} 
\label{sec:strong}

As noted in Section~\ref{sec:outline}, if there are no tight cliques of size at least two, then $G$ satisfies the strong inequalities: for all $t\geq 3$,
\begin{equation}
	\k_t(G)\leq \frac{r-t+1}t \k_{t-1}(G).\label{eqn:strong}
\end{equation}
Note that cliques of size one are tight exactly if the vertex has degree $r$.  Our bound on $k_2(G)=e(G)$ will be the obvious one arising from the degree bound.  The next lemma summarizes these inequalities into a bound on $k(G)$.

\begin{lemma}\label{lem:strongimps}
	If $G$ is a graph on $n$ vertices with $\Delta(G)\leq r$, where $r\geq 2$, and $G$ satisfies the strong inequalities for $t\geq 3$, then
	\[
		\k(G)\leq 1 + \frac{n}{r-1}(2^r-2).
	\]
\end{lemma}

\begin{proof}
	First note that $\k_0(G)=1$, $\k_1(G)=n$, and $\k_2(G)\leq \frac{nr}2$.  For $t\geq 3$, we 
note that, by induction using (\ref{eqn:strong}),
\[
	\k_t(G)\leq \frac{n}{r-1}\binom{r}t.
\]
Hence,
\begin{align*}
	\k(G)&\leq 1+n+\frac{nr}2+\frac{n}{r-1}\left(2^r-\binom{r}2-r-1\right)\\
	&=1+\frac{n}{r-1}(2^r-2).\qedhere
\end{align*}
\end{proof}

\begin{lemma}\label{lem:strong}
	If $r\geq 3$ and $n=a(r+1)+b$ where $a\geq 1$ and $0\leq b\leq r$, then
	\begin{equation}
		1 + \frac{n}{r-1}(2^r-2)\leq a(2^{r+1}-1)+2^b,\label{eq:strong}
	\end{equation}
	with strict inequality unless $r=3$ and $n=6$.
\end{lemma}

\begin{proof}
	It will be convenient to address first the case when $a=1$ and $b=0$.  In this case, we are claiming 
	\begin{equation}
		\frac{r+1}{r-1}(2^{r}-2)<2^{r+1}-1,\label{eq:blah}
	\end{equation}
	which is true for all $r\geq 3$.  
	
	In general, the two sides of (\ref{eq:strong}) are each linear in $a$, with respective coefficients the two sides of (\ref{eq:blah}).  Thus, it suffices to prove the result for $a=1$.  In this case, we need to show
	\begin{align*}
		1+\frac{r+1+b}{r-1}(2^r-2)&< 2^{r+1}-1+2^b\\
		\shortintertext{which is equivalent to}
		(r-1)(2-2^b)-2(r+1+b)&<(r-3-b)2^r.
	\end{align*}
	The last inequality is clearly true if $b\leq r-3$, since the left hand side is negative and the right hand side is nonnegative.  For the remaining cases, i.e., $b=r-2, r-1, r$, we need to check whether
	\[
		(b-(r-3))2^r<(r-1)(2^b-2)+2(r+1+b).
	\]
	This is straightforward to check in each case when $r\geq 5$ and easy to check in the other cases.  In the case $r=3$ and $b=2$, we get equality.
\end{proof}

\begin{corollary}\label{cor:strong}
	Let $G$ be a graph on $n$ vertices with $\Delta(G)\leq r$ which satisfies the strong inequalities for $t\geq 3$.  If $n=a(r+1)+b$ with $a\geq 1$ and $0\leq b\leq r$, then 
	\[
		\k(G)<\k(aK_{r+1}\cup K_b).
	\] 
\end{corollary}

\begin{proof}
	This is immediate from Lemmas~\ref{lem:strongimps} and \ref{lem:strong} except when $n=6$ and $r=3$.  In this case, $\k(G)$ could only be as large as the left hand side of (\ref{eq:blah}) if $G$ were $3$-regular.  Neither of the $3$-regular graphs on $6$ vertices achieves $\k(G)=19$.
\end{proof}


\section{Discharging} 
\label{sec:discharging}

In this section, we discuss the case of the argument wherein every tight clique has large fixed loss.  Throughout this section, we let $G$ be a graph on $n$ vertices with $\Delta(G)\leq r$ and $T$ be a tight clique in $G$.  We set $t=\abs{T}$, $S=S_T$, $R=R_T$, and $s=\abs{S}$.  In order to understand the structure of tight cliques, we make the following definition.

\begin{definition}
	A \emph{cluster} is a maximal tight clique.  If $T$ is a cluster in a graph $G$ and $C$ is a $c$-clique with $\abs{T\cap C}=c-1$, then we say that $C$ is \emph{associated with $T$}.
\end{definition}

Note that $x$ and $y$ belong to some common tight clique exactly if $N[x]=N[y]$.  In particular, the relation of belonging to some common tight clique is an equivalence relation, with clusters as the equivalence classes.

We note a consequence of Corollary~\ref{cor:fli}.

\begin{lemma}\label{lem:sbound}
	If $T$ is a cluster in $G$ with $\k(G_T)\leq \k(G)$, then $\fl(R)\geq 2^r+s2^t$ and $t< \log_2(s)$.
\end{lemma}

\begin{proof}
	From Corollary~\ref{cor:fli}, we know that if $\k(G_T)\leq \k(G)$, then $2^t\leq \fl(R)/(2^s-i(R)+s+1)$.  Since $T$ is a cluster, we have $\delta(R)\geq 1$.  By a result of Galvin \cite{G}, since $\delta(R)\geq 1$, we know that $i(R)\leq 2^{s-1}+1$.  Hence, 
	\[
		\fl(R)\geq 2^t(2^s-i(R)+s+1) \geq 2^t(2^{s-1}+s)=2^r+s2^t.
	\]
	Also, by Theorem~\ref{thm:maxFL},
	\begin{align*}
		t&\leq \log_2\frac{\fl(R)}{2^s-i(R)}\\
		&\leq \log_2\frac{s(2^{s-1}-1)}{2^{s-1}+s}\\
		&< \log_2 s.\qedhere
	\end{align*}
\end{proof}

The main result of this section shows that if the fixed loss of a cluster is large, then there are many cliques associated with that cluster having low weight.  In the proof of the Main Theorem we will transfer weight from tight cliques inside a given cluster to cliques of low weight associated with that cluster, proving that $G$ satisfies the strong inequalities.

\begin{lemma}\label{lem:clusnum}
	Let $r\geq 3$.  If $T$ is a cluster in $G$, $\k(G_T)\leq \k(G)$, and $R$ has no $K_2$ component, then for every $2\leq c\leq t$, there are at least $2\binom{t}{c}$ $c$-cliques associated with $T$ having weight at most $r-c-1$.
\end{lemma}

\begin{proof}
	If $C$ is a $c$-clique associated with $T$ and $x$ is the unique element of $S\cap C$, then $\w(C)=r+1-c-d_R(x)$.  Thus, all associated cliques containing vertices of degree at least two in $R$ have weight at most $r-c-1$.  We will show that there are at least $t-1$ such vertices.  We let $\ell$ be the number of vertices of $R$ of degree one.  If $\ell\geq s-t+2$, then by Theorem~\ref{thm:flell}, we would have $\fl(R)\leq 2^s+(t-4)2^{t-3}$.  Note that this would imply
	\[
		2^r\leq 2^r+s2^t\leq \fl(R)\leq 2^{r+1-t}+(t-4)2^{t-3}\leq 2^{r+1-t}+\frac{1}8 s\log_2 s\leq 2^{r+1-t}+\frac{1}8 r\log_2 r,
	\]
	by Lemma~\ref{lem:sbound}.  Since $t\geq 2$, we see
	\[
		2^{r}\leq 2^{r+1-t}+\frac{1}8 r\log_2 r \quad\implies\quad 2^{r}\leq \frac{1}{4}r\log_2 r,
	\]
	a contradiction for $r\geq 3$.
	
	Let $h$ be the number of vertices in $R$ of degree at least two.  Having shown that $\ell\leq s-t+1$, we know that $h\geq t-1$.  The number of $c$-cliques of weight at most $r-c-1$ associated with $T$ is
	\[
		h\binom{t}{c-1}\geq (t-1)\binom{t}{c-1}\geq \frac{2(t-c+1)}{c}\binom{t}{c-1}=2\binom{t}c,
	\]
	for $c\geq 2$.
\end{proof}


\section{Proof of the Main Theorem} 
\label{sec:proof}

\begin{main}
	For all $n,r\in \N$, write $n=a(r+1)+b$ with $0\le b\le r$. If $G$ is a graph on $n$ vertices with $\Delta(G)\leq r$, then 
	\begin{equation}
		\k(G) \leq \k(aK_{r+1} \cup K_b),\label{eqn:theone}
	\end{equation}
	with equality if and only if $G=aK_{r+1}\cup K_b$, or $r=2$ and $G=(a-1)K_3\cup C_4$ or $(a-1)K_3\cup C_5$.
\end{main}

\begin{proof}
	If $r=1$, the result is trivial.  If $r=2$, it is almost as trivial: for $n=3a+b$ with $0\leq b\leq 2$, we have $e(G)\leq e(aK_3\cup K_b)+1$ with equality if and only if $b\neq 0$ and $G$ is $2$-regular.  Also, $k_3(G)\leq a$.  However, one cannot have equality in both bounds.  The graphs described in the statement of the theorem are the only examples achieving equality in (\ref{eqn:theone}).
	
	We assume henceforth that $r\geq 3$.  We proceed by induction on $n$, noting that the result is trivial for $n< r+1$, i.e., $a=0$.  Consider then a graph $G$ as in the statement of the theorem with $\k(G)$ maximal.  If $K_{r+1}\of G$, then we are done by induction.  If $G$ has no tight cliques of size at least two, then $G$ satisfies the strong inequalities for $t\geq 3$ and, so by Corollary~\ref{cor:strong}, we have $\k(G)<\k(aK_{r+1} \cup K_b)$, a contradiction to the choice of $G$.  
	
	The remaining cases involve graphs $G$ containing tight cliques of size at least two.  We note that $G$ cannot contain a tight clique $T$ with $\k(G_T)>\k(G)$ (by the maximality of $\k(G)$) nor can it contain a tight clique $T$ such that $R_T$ has a $K_2$ component (by Lemma~\ref{lem:k2}).
	
	Thus, $G$ has some tight cliques of size at least two, all of which have $\k(G_T)\leq \k(G)$.  In this final case, we will use the results of Section~\ref{sec:discharging} to show that, in fact, $G$ satisfies the strong inequalities.  We will define new weights on all cliques in the following fashion.  We will reduce the weights of tight cliques by one.  If a clique $C$ is associated with a cluster, we increase its weight by one half.  It is possible for a clique of size two to be associated with two clusters; in this case, we increase its weight by one.  Larger cliques cannot be associated with more than one cluster since they must intersect each cluster in $c-1$ vertices, hence the clusters themselves would intersect.  Denoting these new weights by $w'(C)$, we observe first that Lemma~\ref{lem:clusnum} implies $\sum_{C\in \K_{t}(G)} \w(C)\leq \sum_{C\in \K_{t}(G)} w'(C)$ for $t\geq 2$.  Also, for all $c$-cliques $C$, we have $w'(C)\leq r-c$.  Using (\ref{eq:kbound}), we get
	\[
		k_t(G)=\frac{1}{t}\sum_{C\in \K_{t-1}(G)} \w(C)\leq \frac{1}{t}\sum_{C\in \K_{t-1}(G)} w'(C) \leq \frac{r-t+1}{t} \k_{t-1}(G),
	\]
	for $t\geq 3$, i.e., $G$ satisfies the strong inequalities.
\end{proof}


\bibliographystyle{amsplain}
\bibliography{maxcliquedegree}

\end{document}